\documentclass[10pt]{article}
\pdfoutput=1 
\usepackage{amsthm,amssymb, epsfig, epstopdf, tikz, tikz-cd, tcolorbox, subcaption, setspace, graphicx, hyperref, enumitem, mathrsfs, bbm, mathtools,color,stmaryrd, comment, fancyheadings, mathpazo, euler, subcaption,cleveref}
\theoremstyle{definition}
\usepackage[margin=.9in]{geometry}
\newtheorem{remark}{Remark}[section]
\newtheorem{notation}{Notation}[section]

\newtheorem{theorem}{Theorem}
\newtheorem{definition}{Definition}[section]
\newtheorem{prop}{Proposition}[section]
\newtheorem{lemma}{Lemma}[section]

\newtheorem{fact}{Fact}[section]
\newtheorem{example}{Example}[section]

\begin{document}
\title{Open Maps Preserve Stability} 
\author{James Schmidt\footnote{Department of Applied Mathematics and Statistics,  aschmi40@jhu.edu}\\
\small Johns Hopkins University}
\date{\today}
\maketitle
\newcommand{\nat}[6][large]{%
  \begin{tikzcd}[ampersand replacement = \&, column sep=#1]
    #2\ar[bend left=40,""{name=U}]{r}{#4}\ar[bend right=40,',""{name=D}]{r}{#5}\& #3
          \ar[shorten <=10pt,shorten >=10pt,Rightarrow,from=U,to=D]{d}{~#6}
    \end{tikzcd}
}
\newcommand{\invamalg}{\mathbin{\text{\rotatebox[origin=c]{180}{$\amalg$}}}}

\newcommand{\dCrl}[0]{\mathfrak{dCrl}}
\newcommand{\ytil}{\tilde{y}}
\newcommand{\defeq}{\vcentcolon=}
\newcommand{\dee}{\partial}
\newcommand{\lb}{\{}
\newcommand{\rb}{\}}
\newcommand{\R}{\mathbb{R}}
\newcommand{\C}{\mathbb{C}}
\newcommand{\Q}{\mathbb{Q}}
\newcommand{\N}{\mathbb{N}}
\newcommand{\el}{\mathcal{L}}
\newcommand{\pdiv}[2]{\frac{\partial{#1}}{\partial{#2}}}
\newcommand{\discatp}{\displaystyle\bigsqcap}
\newcommand{\discats}{\displaystyle\bigsqcup}

\newcommand{\uZ}{\underline{\mathbb{Z}}}
\newcommand{\uF}[1]{\underline{\mathbb{F}}}
\newcommand{\one}{\mathbb{1}}
\newcommand{\two}{\mathbb{2}}

\newcommand{\dis}{\displaystyle}
\newcommand{\disp}{\displaystyle\prod}
\newcommand{\disu}{\displaystyle\bigcup}
\newcommand{\disi}{\displaystyle\bigcap}
\newcommand{\diss}{\displaystyle\sum}
\newcommand{\disg}{\displaystyle\int}
\newcommand{\disl}{\displaystyle\lim}
\newcommand{\dislim}{\displaystyle\lim}
\newcommand{\disliminf}{\displaystyle\liminf}
\newcommand{\dislimsup}{\displaystyle\limsup}
\newcommand{\disbop}{\displaystyle\bigotimes}
\newcommand{\disbos}{\displaystyle\bigoplus}
\newcommand{\dissup}{\displaystyle\sup}
\newcommand{\disinf}{\displaystyle\inf}
\newcommand{\dismax}{\displaystyle\max}
\newcommand{\dismin}{\displaystyle\min}
\newcommand{\dirlim}[1]{\displaystyle\varinjlim_{#1}}
\newcommand{\indlim}[1]{\displaystyle\varprojlim_{#1}}
\newcommand{\discatlim}{\indlim}
\newcommand{\discatcolim}{\dirlim}
\newcommand{\catcolim}{\mbox{colim}}
\newcommand{\catlim}{\mbox{lim}}

\newcommand{\colgray}[1]{\color{gray}{#1}\color{black}}

\newcommand{\sfM}{\sF{M}}

\newcommand{\forget}[2]{\Ub^{#1}_{#2}}

\newcommand\righttwoarrow{%
        \mathrel{\vcenter{\mathsurround0pt
                \ialign{##\crcr
                        \noalign{\nointerlineskip}$\rightarrow$\crcr
                        \noalign{\nointerlineskip}$\rightarrow$\crcr
                }%
        }}%
}

\newcommand{\Z}{\mathbb{Z}}
\newcommand{\Ab}[0]{\mathbb{A}}
\newcommand{\Bb}[0]{\mathbb{B}}
\newcommand{\Cb}[0]{\mathbb{C}}
\newcommand{\Db}[0]{\mathbb{D}}
\newcommand{\Eb}[0]{\mathbb{E}}
\newcommand{\Fb}[0]{\mathbb{F}}
\newcommand{\Gb}[0]{\mathbb{G}}
\newcommand{\Hb}[0]{\mathbb{H}}
\newcommand{\Ib}[0]{\mathbb{I}}
\newcommand{\Jb}[0]{\mathbb{J}}
\newcommand{\Kb}[0]{\mathbb{K}}
\newcommand{\Lb}[0]{\mathbb{L}}
\newcommand{\Mb}[0]{\mathbb{M}}
\newcommand{\Nb}[0]{\mathbb{N}}
\newcommand{\Ob}[0]{\mathbb{O}}
\newcommand{\Pb}[0]{\mathbb{P}}
\newcommand{\Qb}[0]{\mathbb{Q}}
\newcommand{\Rb}[0]{\mathbb{R}}
\newcommand{\Sb}[0]{\mathbb{S}}
\newcommand{\Tb}[0]{\mathbb{T}}
\newcommand{\Ub}[0]{\mathbb{U}}
\newcommand{\Vb}[0]{\mathbb{V}}
\newcommand{\Wb}[0]{\mathbb{W}}
\newcommand{\Xb}[0]{\mathbb{X}}
\newcommand{\Yb}[0]{\mathcal{Y}}
\newcommand{\Zb}[0]{\mathbb{Z}}

\newcommand{\sF}[1]{\mathsf{#1}}

\newcommand{\sC}[1]{\mathscr{#1}}

\newcommand{\mC}[1]{\mathcal{#1}}

\newcommand{\mB}[1]{\mathbb{#1}}

\newcommand{\mF}[1]{\mathfrak{#1}}

\newcommand{\Bc}[0]{\mathcal{B}}
\newcommand{\Cc}[0]{\mathcal{C}}
\newcommand{\Dc}[0]{\mathcal{D}}
\newcommand{\Ec}[0]{\mathcal{E}}
\newcommand{\Fc}[0]{\mathcal{F}}
\newcommand{\Gc}[0]{\mathcal{G}}
\newcommand{\Hc}[0]{\mathcal{H}}
\newcommand{\Ic}[0]{\mathcal{I}}
\newcommand{\Jc}[0]{\mathcal{J}}
\newcommand{\Kc}[0]{\mathcal{K}}
\newcommand{\Lc}[0]{\mathcal{L}}
\newcommand{\Mc}[0]{\mathcal{M}}
\newcommand{\Nc}[0]{\mathcal{N}}
\newcommand{\Oc}[0]{\mathcal{O}}
\newcommand{\Pc}[0]{\mathcal{P}}
\newcommand{\Qc}[0]{\mathcal{Q}}
\newcommand{\Rc}[0]{\mathcal{R}}
\newcommand{\Sc}[0]{\mathcal{S}}
\newcommand{\Tc}[0]{\mathcal{T}}
\newcommand{\Uc}[0]{\mathcal{U}}
\newcommand{\Vc}[0]{\mathcal{V}}
\newcommand{\Wc}[0]{\mathcal{W}}
\newcommand{\Xc}[0]{\mathcal{X}}
\newcommand{\Yc}[0]{\mathcal{Y}}
\newcommand{\Zc}[0]{\mathcal{Z}}

\newcommand{\aca}[0]{\mathcal{a}}
\newcommand{\bca}[0]{\mathcal{b}}
\newcommand{\cca}[0]{\mathcal{c}}
\newcommand{\dca}[0]{\mathcal{d}}
\newcommand{\eca}[0]{\mathcal{e}}
\newcommand{\fca}[0]{\mathcal{f}}
\newcommand{\gca}[0]{\mathcal{g}}
\newcommand{\hca}[0]{\mathcal{h}}
\newcommand{\ica}[0]{\mathcal{i}}
\newcommand{\jca}[0]{\mathcal{j}}
\newcommand{\kca}[0]{\mathcal{k}}
\newcommand{\lca}[0]{\mathcal{l}}
\newcommand{\mca}[0]{\mathcal{m}}
\newcommand{\nca}[0]{\mathcal{n}}
\newcommand{\oca}[0]{\mathcal{o}}
\newcommand{\pca}[0]{\mathcal{p}}
\newcommand{\qca}[0]{\mathcal{q}}
\newcommand{\rca}[0]{\mathcal{r}}
\newcommand{\sca}[0]{\mathcal{s}}
\newcommand{\tca}[0]{\mathcal{t}}
\newcommand{\uca}[0]{\mathcal{u}}
\newcommand{\vca}[0]{\mathcal{v}}
\newcommand{\wca}[0]{\mathcal{w}}
\newcommand{\xca}[0]{\mathcal{x}}
\newcommand{\yca}[0]{\mathcal{y}}
\newcommand{\zca}[0]{\mathcal{z}}

\newcommand{\Af}[0]{\mathfrak{A}}
\newcommand{\Bf}[0]{\mathfrak{B}}
\newcommand{\Cf}[0]{\mathfrak{C}}
\newcommand{\Df}[0]{\mathfrak{D}}
\newcommand{\Ef}[0]{\mathfrak{E}}
\newcommand{\Ff}[0]{\mathfrak{F}}
\newcommand{\Gf}[0]{\mathfrak{G}}
\newcommand{\Hf}[0]{\mathfrak{H}}
\newcommand{\If}[0]{\mathfrak{I}}
\newcommand{\Jf}[0]{\mathfrak{J}}
\newcommand{\Kf}[0]{\mathfrak{K}}
\newcommand{\Lf}[0]{\mathfrak{L}}
\newcommand{\Mf}[0]{\mathfrak{M}}
\newcommand{\Nf}[0]{\mathfrak{N}}
\newcommand{\Of}[0]{\mathfrak{O}}
\newcommand{\Pf}[0]{\mathfrak{P}}
\newcommand{\Qf}[0]{\mathfrak{Q}}
\newcommand{\Rf}[0]{\mathfrak{R}}
\newcommand{\Sf}[0]{\mathfrak{S}}
\newcommand{\Tf}[0]{\mathfrak{T}}
\newcommand{\Uf}[0]{\mathfrak{U}}
\newcommand{\Vf}[0]{\mathfrak{V}}
\newcommand{\Wf}[0]{\mathfrak{W}}
\newcommand{\Xf}[0]{\mathfrak{X}}
\newcommand{\Yf}[0]{\mathfrak{Y}}
\newcommand{\Zf}[0]{\mathfrak{Z}}

\newcommand{\af}[0]{\mathfrak{a}}
\newcommand{\bff}[0]{\mathfrak{b}}
\newcommand{\cf}[0]{\mathfrak{c}}
\newcommand{\dff}[0]{\mathfrak{d}}
\newcommand{\ef}[0]{\mathfrak{e}}
\newcommand{\ff}[0]{\mathfrak{f}}
\newcommand{\gf}[0]{\mathfrak{g}}
\newcommand{\hf}[0]{\mathfrak{h}}
\newcommand{\ifrak}{\mathfrak{i}}
\newcommand{\jf}[0]{\mathfrak{j}}
\newcommand{\kf}[0]{\mathfrak{k}}
\newcommand{\lf}[0]{\mathfrak{l}}
\newcommand{\mf}[0]{\mathfrak{m}}
\newcommand{\nf}[0]{\mathfrak{n}}
\newcommand{\of}[0]{\mathfrak{o}}
\newcommand{\pf}[0]{\mathfrak{p}}
\newcommand{\qf}[0]{\mathfrak{q}}
\newcommand{\rf}[0]{\mathfrak{r}}
\renewcommand{\sf}[0]{\mathfrak{s}}
\newcommand{\tf}[0]{\mathfrak{t}}
\newcommand{\uf}[0]{\mathfrak{u}}
\newcommand{\vf}[0]{\mathfrak{v}}
\newcommand{\wf}[0]{\mathfrak{w}}
\newcommand{\xf}[0]{\mathfrak{x}}
\newcommand{\yf}[0]{\mathfrak{y}}
\newcommand{\zf}[0]{\mathfrak{z}}
\newcommand{\cdX}{\mathfrak{cdX}}
\newcommand{\cdCrl}{\mathfrak{cdCrl}}

\newcommand{\scA}[0]{\mathscr{A}}
\newcommand{\scB}[0]{\mathscr{B}}
\newcommand{\scC}[0]{\mathscr{C}}
\newcommand{\scD}[0]{\mathscr{D}}
\newcommand{\scE}[0]{\mathscr{E}}
\newcommand{\scF}[0]{\mathscr{F}}
\newcommand{\scG}[0]{\mathscr{G}}
\newcommand{\scH}[0]{\mathscr{H}}
\newcommand{\scI}[0]{\mathscr{I}}
\newcommand{\scJ}[0]{\mathscr{J}}
\newcommand{\scK}[0]{\mathscr{K}}
\newcommand{\scL}[0]{\mathscr{L}}
\newcommand{\scM}[0]{\mathscr{M}}
\newcommand{\scN}[0]{\mathscr{N}}
\newcommand{\scO}[0]{\mathscr{O}}
\newcommand{\scP}[0]{\mathscr{P}}
\newcommand{\scQ}[0]{\mathscr{Q}}
\newcommand{\scR}[0]{\mathscr{R}}
\newcommand{\scS}[0]{\mathscr{S}}
\newcommand{\scT}[0]{\mathscr{T}}
\newcommand{\scU}[0]{\mathscr{U}}
\newcommand{\scV}[0]{\mathscr{V}}
\newcommand{\scW}[0]{\mathscr{W}}
\newcommand{\scX}[0]{\mathscr{X}}
\newcommand{\scY}[0]{\mathscr{Y}}
\newcommand{\scZ}[0]{\mathscr{Z}}

\newcommand{\fA}[0]{\mathsf{A}}
\newcommand{\fB}[0]{\mathsf{B}}
\newcommand{\fC}[0]{\mathsf{C}}
\newcommand{\fD}[0]{\mathsf{D}}
\newcommand{\fE}[0]{\mathsf{E}}
\newcommand{\fG}[0]{\mathsf{G}}
\newcommand{\fH}[0]{\mathsf{H}}
\newcommand{\fI}[0]{\mathsf{I}}
\newcommand{\fJ}[0]{\mathsf{J}}
\newcommand{\fK}[0]{\mathsf{K}}
\newcommand{\fL}[0]{\mathsf{L}}
\newcommand{\fM}[0]{\mathsf{M}}
\newcommand{\fN}[0]{\mathsf{N}}
\newcommand{\fO}[0]{\mathsf{O}}
\newcommand{\fP}[0]{\mathsf{P}}
\newcommand{\fQ}[0]{\mathsf{Q}}
\newcommand{\fR}[0]{\mathsf{R}}
\newcommand{\fS}[0]{\mathsf{S}}
\newcommand{\fT}[0]{\mathsf{T}}
\newcommand{\fU}[0]{\mathsf{U}}
\newcommand{\fV}[0]{\mathsf{V}}
\newcommand{\fW}[0]{\mathsf{W}}
\newcommand{\fX}[0]{\mathsf{X}}
\newcommand{\fY}[0]{\mathsf{Y}}
\newcommand{\fZ}[0]{\mathsf{Z}}

\newcommand{\fa }[0]{\mathsf{a}}
\newcommand{\fb }[0]{\mathsf{b}}
\newcommand{\fc }[0]{\mathsf{c}}
\newcommand{\fd}[0]{\mathsf{d}}
\newcommand{\fe}[0]{\mathsf{e}}
\newcommand{\fg}[0]{\mathsf{g}}
\newcommand{\fh}[0]{\mathsf{h}}
\newcommand{\fj}[0]{\mathsf{j}}
\newcommand{\fk}[0]{\mathsf{k}}
\newcommand{\fl }[0]{\mathsf{l}}
\newcommand{\fm }[0]{\mathsf{m}}
\newcommand{\fn }[0]{\mathsf{n}}
\newcommand{\fo }[0]{\mathsf{o}}
\newcommand{\fp}[0]{\mathsf{p}}
\newcommand{\fq}[0]{\mathsf{q}}
\newcommand{\fr}[0]{\mathsf{r}}
\newcommand{\fs}[0]{\mathsf{s}}
\newcommand{\ft }[0]{\mathsf{t}}
\newcommand{\fu }[0]{\mathsf{u}}
\newcommand{\fv }[0]{\mathsf{v}}
\newcommand{\fw}[0]{\mathsf{w}}
\newcommand{\fx}[0]{\mathsf{x}}
\newcommand{\fy}[0]{\mathsf{y}}
\newcommand{\fz}[0]{\mathsf{z}}

\newcommand{\dA}[0]{\dot{A}}
\newcommand{\dB}[0]{\dot{B}}
\newcommand{\dC}[0]{\dot{C}}
\newcommand{\dD}[0]{\dot{D}}
\newcommand{\dE}[0]{\dot{E}}
\newcommand{\dF}[0]{\dot{F}}
\newcommand{\dG}[0]{\dot{G}}
\newcommand{\dH}[0]{\dot{H}}
\newcommand{\dI}[0]{\dot{I}}
\newcommand{\dJ}[0]{\dot{J}}
\newcommand{\dK}[0]{\dot{K}}
\newcommand{\dL}[0]{\dot{L}}
\newcommand{\dM}[0]{\dot{M}}
\newcommand{\dN}[0]{\dot{N}}
\newcommand{\dO}[0]{\dot{O}}
\newcommand{\dP}[0]{\dot{P}}
\newcommand{\dQ}[0]{\dot{Q}}
\newcommand{\dR}[0]{\dot{R}}
\newcommand{\dS}[0]{\dot{S}}
\newcommand{\dT}[0]{\dot{T}}
\newcommand{\dU}[0]{\dot{U}}
\newcommand{\dV}[0]{\dot{V}}
\newcommand{\dW}[0]{\dot{W}}
\newcommand{\dX}[0]{\dot{X}}
\newcommand{\dY}[0]{\dot{Y}}
\newcommand{\dZ}[0]{\dot{Z}}

\newcommand{\da}[0]{\dot{a}}
\newcommand{\db}[0]{\dot{b}}
\newcommand{\dc}[0]{\dot{c}}
\newcommand{\dd}[0]{\dot{d}}
\newcommand{\de}[0]{\dot{e}}
\newcommand{\df}[0]{\dot{f}}
\newcommand{\dg}[0]{\dot{g}}
\renewcommand{\dh}[0]{\dot{h}}
\newcommand{\di}[0]{\dot{i}}
\renewcommand{\dj}[0]{\dot{j}}
\newcommand{\dk}[0]{\dot{k}}
\newcommand{\dl}[0]{\dot{l}}
\newcommand{\dm}[0]{\dot{m}}
\newcommand{\dn}[0]{\dot{n}}
\newcommand{\dq}[0]{\dot{q}}
\newcommand{\dr}[0]{\dot{r}}
\newcommand{\ds}[0]{\dot{s}}
\newcommand{\dt}[0]{\dot{t}}
\newcommand{\du}[0]{\dot{u}}
\newcommand{\dv}[0]{\dot{v}}
\newcommand{\dw}[0]{\dot{w}}
\newcommand{\dx}[0]{\dot{x}}
\newcommand{\dy}[0]{\dot{y}}
\newcommand{\dz}[0]{\dot{z}}

\newcommand{\oA}[0]{\overline{A}}
\newcommand{\oB}[0]{\overline{B}}
\newcommand{\oC}[0]{\overline{C}}
\newcommand{\oD}[0]{\overline{D}}
\newcommand{\oE}[0]{\overline{E}}
\newcommand{\oF}[0]{\overline{F}}
\newcommand{\oG}[0]{\overline{G}}
\newcommand{\oH}[0]{\overline{H}}
\newcommand{\oI}[0]{\overline{I}}
\newcommand{\oJ}[0]{\overline{J}}
\newcommand{\oK}[0]{\overline{K}}
\newcommand{\oL}[0]{\overline{L}}
\newcommand{\oM}[0]{\overline{M}}
\newcommand{\oN}[0]{\overline{N}}
\newcommand{\oO}[0]{\overline{O}}
\newcommand{\oP}[0]{\overline{P}}
\newcommand{\oQ}[0]{\overline{Q}}
\newcommand{\oR}[0]{\overline{R}}
\newcommand{\oS}[0]{\overline{S}}
\newcommand{\oT}[0]{\overline{T}}
\newcommand{\oU}[0]{\overline{U}}
\newcommand{\oV}[0]{\overline{V}}
\newcommand{\oW}[0]{\overline{W}}
\newcommand{\oX}[0]{\overline{X}}
\newcommand{\oY}[0]{\overline{Y}}
\newcommand{\oZ}[0]{\overline{Z}}

\newcommand{\oa}[0]{\overline{a}}
\newcommand{\ob}[0]{\overline{b}}
\newcommand{\oc}[0]{\overline{c}}
\newcommand{\od}[0]{\overline{d}}
\renewcommand{\oe}[0]{\overline{e}}
\newcommand{\og}[0]{\overline{g}}
\newcommand{\oh}[0]{\overline{h}}
\newcommand{\oi}[0]{\overline{i}}
\newcommand{\oj}[0]{\overline{j}}
\newcommand{\ok}[0]{\overline{k}}
\newcommand{\ol}[0]{\overline{l}}
\newcommand{\om}[0]{\overline{m}}
\newcommand{\on}[0]{\overline{n}}
\newcommand{\oo}[0]{\overline{o}}
\newcommand{\op}[0]{\overline{p}}
\newcommand{\oq}[0]{\overline{q}}
\newcommand{\os}[0]{\overline{s}}
\newcommand{\ot}[0]{\overline{t}}
\newcommand{\ou}[0]{\overline{u}}
\newcommand{\ov}[0]{\overline{v}}
\newcommand{\ow}[0]{\overline{w}}
\newcommand{\ox}[0]{\overline{x}}
\newcommand{\oy}[0]{\overline{y}}
\newcommand{\oz}[0]{\overline{z}}

\renewcommand{\a}{\alpha}
\renewcommand{\b}{\beta}
\renewcommand{\d}{\delta}
\newcommand{\e}{\varepsilon}
\newcommand{\f}{\phi}
\newcommand{\g}{\gamma}
\newcommand{\h}{\eta}
\renewcommand{\i}{\iota}
\renewcommand{\k}{\kappa}
\renewcommand{\l}{\lambda}
\newcommand{\m}{\mu}
\newcommand{\n}{\nu}
\newcommand{\p}{\pi}
\newcommand{\ph}{\varphi}
\newcommand{\ps}{\psi}
\newcommand{\q}{\xi}
\renewcommand{\r}{\rho}
\newcommand{\s}{\sigma}
\renewcommand{\t}{\tau}
\renewcommand{\v}{\upsilon}
\newcommand{\x}{\chi}
\newcommand{\z}{\zeta}
\newcommand{\G}{\Gamma}

\newcommand{\aarb}[0]{\<a>}
\newcommand{\barb}[0]{\<b>}
\newcommand{\carb}[0]{\<c>}
\newcommand{\darb}[0]{\<d>}
\newcommand{\earb}[0]{\<e>}
\newcommand{\farb}[0]{\<f>}
\newcommand{\garb}[0]{\<g>}
\newcommand{\harb}[0]{\<h>}
\newcommand{\iarb}[0]{\<i>}
\newcommand{\jarb}[0]{\<j>}
\newcommand{\karb}[0]{\<k>}
\newcommand{\larb}[0]{\<l>}
\newcommand{\marb}[0]{\<m>}
\newcommand{\narb}[0]{\<n>}
\newcommand{\oarb}[0]{\<o>}
\newcommand{\parb}[0]{\<p>}
\newcommand{\qarb}[0]{\<q>}
\newcommand{\rarb}[0]{\<r>}
\newcommand{\sarb}[0]{\<s>}
\newcommand{\tarb}[0]{\<t>}
\newcommand{\uarb}[0]{\<u>}
\newcommand{\varb}[0]{\<v>}
\newcommand{\warb}[0]{\<w>}
\newcommand{\xarb}[0]{\<x>}
\newcommand{\yarb}[0]{\<y>}
\newcommand{\zarb}[0]{\<z>}

\newcommand{\hA}[0]{\hat{A}}
\newcommand{\hB}[0]{\hat{B}}
\newcommand{\hC}[0]{\hat{C}}
\newcommand{\hD}[0]{\hat{D}}
\newcommand{\hE}[0]{\hat{E}}
\newcommand{\hF}[0]{\hat{F}}
\newcommand{\hG}[0]{\hat{G}}
\newcommand{\hH}[0]{\hat{H}}
\newcommand{\hI}[0]{\hat{I}}
\newcommand{\hJ}[0]{\hat{J}}
\newcommand{\hK}[0]{\hat{K}}
\newcommand{\hL}[0]{\hat{L}}
\newcommand{\hM}[0]{\hat{M}}
\newcommand{\hN}[0]{\hat{N}}
\newcommand{\hO}[0]{\hat{O}}
\newcommand{\hP}[0]{\hat{P}}
\newcommand{\hQ}[0]{\hat{Q}}
\newcommand{\hR}[0]{\hat{R}}
\newcommand{\hS}[0]{\hat{S}}
\newcommand{\hT}[0]{\hat{T}}
\newcommand{\hU}[0]{\hat{U}}
\newcommand{\hV}[0]{\hat{V}}
\newcommand{\hW}[0]{\hat{W}}
\newcommand{\hX}[0]{\hat{X}}
\newcommand{\hY}[0]{\hat{Y}}
\newcommand{\hZ}[0]{\hat{Z}}

\newcommand{\ha}[0]{\hat{a}}
\newcommand{\hb}[0]{\hat{b}}
\newcommand{\hc}[0]{\hat{c}}
\newcommand{\hd}[0]{\hat{d}}
\newcommand{\he}[0]{\hat{e}}
\newcommand{\hg}[0]{\hat{g}}
\newcommand{\hh}[0]{\hat{h}}
\newcommand{\hi}[0]{\hat{i}}
\newcommand{\hj}[0]{\hat{j}}
\newcommand{\hk}[0]{\hat{k}}
\newcommand{\hl}[0]{\hat{l}}
\newcommand{\hm}[0]{\hat{m}}
\newcommand{\hn}[0]{\hat{n}}
\newcommand{\ho}[0]{\hat{o}}
\newcommand{\hp}[0]{\hat{p}}
\newcommand{\hq}[0]{\hat{q}}
\newcommand{\hr}[0]{\hat{r}}
\newcommand{\hs}[0]{\hat{s}}
\newcommand{\hu}[0]{\hat{u}}
\newcommand{\hv}[0]{\hat{v}}
\newcommand{\hw}[0]{\hat{w}}
\newcommand{\hx}[0]{\hat{x}}
\newcommand{\hy}[0]{\hat{y}}
\newcommand{\hz}[0]{\hat{z}}

\newcommand{\hyph}[2]{\Fc_{#1}:#2\rightarrow \sF{RelMan^c}}
\newcommand{\repsys}{(\Fc_{N^\scT}:\N\rightarrow\sF{RelMan^c},\frac{d}{dt})}
\newcommand{\truerep}[1]{\left(\Fc_{N^\scT}{#1}:\N\rightarrow\sF{RelMan^c},\frac{D}{Dt}\right)}

\abstract{Stability is a fundamental notion in dynamical systems and control theory that, traditionally understood,  describes asymptotic behavior of solutions around an equilibrium point.  This notion  may be characterized abstractly as continuity of a map associating to each point in a state-space the corresponding integral curve with specified initial condition. Interpreting stability as such permits a natural perspective of arbitrary trajectories as stable, irrespective of whether they are stationary or even bounded, so long as trajectories starting nearby stay nearby for all time. 

While methods exist for recognizing stability of equilibria points, such as Lyapunov's first and second methods, such rely on the \textit{state's} local property, which may be readily computed or evaluated because solutions starting at equilibria go nowhere. Such methods do not obviously extend for non-stationary stable trajectories. After introducing a notion of stability which makes sense for trajectories generally, we give examples confirming intuition and then present a method for using knowledge of stability of one system to guarantee stability of another, so long as there is an open map of dynamical systems from the known stable system.  We thus define maps of dynamical systems and prove that a class of open maps preserve stability.}


\section{Introduction}

Stability is a property of dynamical systems guaranteeing that behavior of  
trajectories starting near equilibria points remain nearby for the rest of time. Morally, this property is nice to inhere because solutions may be qualitatively understood locally: bumping an equilibria point may induce trajectory disturbances (wiggles), but the perturbation will not blow up.    In control theory, which  generally is in the business of designing vector fields inducing certain behavior in a resulting  trajectory subject to constraints (what the control system allows), often stability  is not merely a nice-to-have but a requirement. For example, one typically seeks to drive an error to zero. Once that goal is achieved the system  is ``at (or near) equilibrium,'' and it may be imperative to have confidence that slight movement off equilibrium---either due to uncertainty in measurement or external disturbances---will not wildly destroy the realization of negligible error. 

In addition to being stable, it is also important to be able to verify \textit{that} a system is stable. Given an explicit solution, one may usually confirm  stability  immediately and directly. Absent solution, there are methods for checking stability from the dynamics itself, such as Lyapunov's first and second methods (\cite{khalil}), alternatively known more descriptively as linearization about the point of interest, and construction of a dissipative-along-solutions energy function, respectively. Linearization is limited when eigenvalues have zero real part, and in many cases there is no constructive algorithm for generating a dissipative energy function. Moreover, these methods are restricted, as far as we know, to stability of an \textit{equilibrium point}, and do not obviously extend for non-stationary trajectories.


 We introduce a notion of stability which makes sense for arbitrary solutions of  dynamical systems. We will confine our attention to continuous-time dynamical system as manifold and vector field pair $(M,X)$, and we  recall the \textit{solution map} which sends a point $x_0\in M$ to the solution of $X$ passing through $x_0$ at time $0$. We then review the standard definition of Lyapunov stability and generalize its definition as continuity of the solution map.    While this abstraction is not new, its appearance and application are sparse in the dynamical systems literature.  We note an immediate upshot, for example,  that the extra generality allows for a natural  description of arbitrary and even unbounded trajectories as stable. In control theory, where one often cares about driving a system to some desired---not necessarily equilibrium---trajectory,  error (deviation of state from desired trajectory) may be and often is used as proxy state for the underlying system.  In this setting, what is sought is that the dynamical system representing error has solutions which go to zero, and moreover,  that zero is a stable equilibrium, thus guaranteeing that a control algorithm is robust with respect to uncertainty, noise, or disturbances.  In this vein, perhaps little seems to be gained from the extra generality. 

However, we are now able to consider stability in the context of \textit{maps} of systems.  It makes sense to speak of composition as preserving continuity, as long as we are careful about working in the right topology.  We discuss a topology in the space of maps of dynamical systems, and use this to prove \cref{theorem:pushStableToStable} which says that an open map between dynamical systems sends bounded stable points to stable  points.  We end with an example illustrating application of this result, by mapping a \textit{linear} system (whose stability properties are entirely known by eigenvalues of the matrix representing  its dynamics) to a nonlinear system. Though the nonlinear system can be explicitly solved for (in particular, simply by pushing the linear solution forward), and therefore stability determined through other means, linearization by itself fails to detect stability.  While continuity is a local concept,  local-in-a-topology-on-$M$ or -$TM$ is different than local-in-the-space-of-maps-of-systems, which explains why \cref{theorem:pushStableToStable} can answer stability questions which linearization of a  vector field cannot.

 In what follows, we loosely situate our formalism for a general notion of dynamical systems in category theoretic language and worldview, and in particular focus on the map-centric philosophy inherent in this framework. The Yoneda embedding justifies  investigating mathematical objects through their (collection of) maps, and we call to mind, in particular,  that existence and uniqueness for complete continuous-time systems can expressed alternatively as representability of a certain functor  through the Grothendieck construction. While we do not heavily rely on the category theory, we present the main result \cref{theorem:pushStableToStable} as a first step toward a workable category theory of dynamical systems. We believe that this result gives credence to the claim that \textit{maps} of dynamical systems may generate additional insight into dynamical systems themselves, in addition to analysis of systems in isolation. 

\section{Review of  Dynamical Systems and Stability}\label{sec:dysys}
We start with a geometric definition of dynamical system. In the discrete case, we think of iterates of maps. In continuous-time, we think of trajectories in space whose velocity satisfies a specified ordinary differential equation. In general, we would like to describe dynamics simply in terms of the ``governing dynamics.'' For what follows, we consider only continuous-time systems: 

\begin{definition}\label{def:ctDySys}
	We define a continuous-time dynamical system to be a pair $(M,X)$ where $M$ is a smooth manifold and $X\in \Xf(M)$ a smooth vector field on $M$. 
\end{definition}

We state at the outset that we take all smooth manifolds to be Hausdorff and second countable. One may suppose  $M=\R^n$ and that $\dx = X(x)$ denotes an ordinary differential equation; we do not require much more generality.

As we want to use information from one system to describe constraints on behavior in another, we present a notion of map \textit{between} two dynamical systems. To do so, we  relate how the map treats tangent vectors in the source. 

\begin{definition}\label{def:relatedVectorFields} Let $f:M\rightarrow N$ be a smooth map of manifolds.  We say that vector fields $X\in \Xf(M)$ and $Y\in \Xf(N)$ are $f$-\textit{related} if $Tf\circ X = Y\circ f$.	
\end{definition}

Apriori, there is no reason why the tangent vector defined at $f(x)$ by $Y$ should agree with the pushforward under $Tf$ of $X(x)$. This definition says that vector fields are map-related when this relation holds. Relatedness is the condition we need for  a notion of map of dynamical systems:

\begin{definition}\label{def:ctDySysMorphism}
	Let $(M,X)$ and $(N,Y)$ be two continuous-time dynamical systems.  A  \textit{map} (or \textit{morphism}) $(M,X)\xrightarrow{f}(N,Y)$ of systems is a smooth map $M\xrightarrow{f}N$ of manifolds such that $(X,Y)$ are $f$-related (\cref{def:relatedVectorFields}). 
\end{definition}

We will isolate a special class of maps of dynamical systems. Before identifying them as such, we recall the standard definition. 
\begin{definition}\label{def:solutionToDynamicalSystem}
Let $(M,X)$ be a continuous-time dynamical system.  A  \textit{solution} $\ph_{X}$ of $(M,X)$---also called an \textit{integral curve}---is a map $\ph_{X}:(-\e,\e)\rightarrow M$, for some $\e>0$, such that $\frac{d}{dt}\ph_{X}(t) = X(\ph_{X}(t))$ for all $t\in (-\e,\e)$. The value $\ph_X(0)=x_0$ at $t=0$ is called the \textit{initial condition}, and we may write $\ph_{X,x_0}$ to indicate  $x_0$ as the initial condition.  

A solution may have non-symmetric domain $(-\d,\e)$,  and we say that $\ph_X$ is maximal if its domain may not be extended, i.e.\ if there is no $(-\d',\e')\supsetneq (-\d,\e)$ for which $\psi_X:(-\d',\e')\rightarrow M$ is an integral curve. \end{definition}

\begin{remark}\label{remark:completeDySys}
	There are systems  whose maximal solution has domain  all of  $\R$.  Such solutions are said to be \textit{complete}. \end{remark}
	
	 Every dynamical system $(M,X)$ has  solutions.  Moreover, solutions are  unique. We recall and restate the central Existence and Uniqueness Theorem: 
	 
\begin{theorem}\label{theorem:E&U}
	Let $(M,X)$ be a dynamical system, and $x_0\in M$. Then there are $\e,\d>0$ for which a smooth map $\ph_{X,x_0}:(-\d,\e)\rightarrow M$ is unique maximal solution of $(M,X)$ with initial condition $x_0$.  Thus,  $\ph_{X,x_0}(0) = 0$ and $\frac{d}{dt} \ph_{X,x_0}(t) = X(\ph_{X,x_0}(t))$ for $t\in (-\d,\e)$.  Moreover, given curve $\g:(-\d',\e')\rightarrow M$ satisfying $\g(0)= x_0$ and $\frac{d}{dt}\g(t) = X(\g(t))$, then $(-\d',\e')\subseteq (-\d,\e)$ and $\g(t) = \ph_X(t)$ for $t\in (-\d',\e')$. 
\end{theorem}
\begin{proof}
	See \cite[\S 14.3]{tu}.
\end{proof}

 An equivalent defintion of integral curves in  the vein of \cref{def:ctDySysMorphism}: 
\begin{definition}\label{def:integralCurveAsMap}
	Let $(M,X)$ be a continuous-time dynamical system.  A \textit{solution}  (or \textit{integral curve}) \textit{of system} $(M,X)$  is a map $\ph_{X,x_0}:((-\e,\e),\frac{d}{dt})\rightarrow (M,X)$ of dynamical systems from the dynamical system $((-\e,\e),\frac{d}{dt})$ with constant vector field $\frac{d}{dt}\in \Xf(\R)$ sending $t\mapsto 1\in T_t\R$. 
\end{definition}
Equivalence of \cref{def:solutionToDynamicalSystem} and \cref{def:integralCurveAsMap} follows from \cref{def:ctDySysMorphism}, since $X\circ \ph_{X,x_0} = T\ph_{X,x_0} \left(\frac{d}{dt}\right) = \frac{d}{dt} \ph_{X,x_0}$.  Recalling \cref{def:ctDySys},  \cref{def:ctDySysMorphism}, and \cref{remark:completeDySys}, we  package   \textit{complete} dynamical systems into a category. 
\begin{definition}\label{def:completeDySys}
	A \textit{complete (continuous-time)  dynamical system} $(M,X)$ is a pair where $M$ is a manifold and $X\in \Xf(M)$ is a smooth vector field on $M$ such that for each initial condition $x_0\in M$, there is a complete integral curve with initial condition $x_0$, i.e.\ a map $\ph_{X,x_0}:\R\rightarrow M$ satisfying $\frac{d}{dt} \ph_{X,x_0}(t) = X(\ph_{X,x_0}(t))$ for all $t\in \R$. 
	
	\textit{Morphisms (maps) of complete dynamical systems} are maps of dynamical systems (\cref{def:ctDySysMorphism}), namely  maps $(M,X)\xrightarrow{f} (N,Y)$, where both $(M,X)$ and $(N,Y)$ are complete.  We denote the collection of all maps between $(M,X)$ and $(N,Y)$ as $\sF{DySys}\big((M,X),(N,Y)\big)$. 
	 
	Collecting objects (complete dynamical systems) and morphisms (maps of complete dynamical systems) defines a category of \textit{complete dynamical systems}, which we denote as $\sF{DySys}$. \end{definition}

\begin{definition}\label{def:solutionMap}
	A complete dynamical system $(M,X)\in \sF{DySys}$ defines map $$\ph_{X,(\cdot)}:M\rightarrow\sF{DySys}\left(\big(\R,\frac{d}{dt}\big),\big(M,X\big)\right)$$  by sending $x_0\mapsto \ph_{X,x_0}(\cdot)$, the integral curve of $(M,X)$ with initial condition $\ph_{X,X_0}(0) = x_0$. We call $\ph_{X,(\cdot)}$ the \textit{solution  map} of $(M,X)$, and $\ph_{X,x_0}(\cdot)$---the solution map evaluated at point $x_0\in M$---the \textit{solution of} $(M,X)$ \textit{with initial condition} $x_0$. 
\end{definition}

When the system $(M,X)$ is fixed or obvious, we may drop  dependence of $\ph_X$ on vector field $X$ and simply write $\ph$. (We may still want to denote the initial condition in the subscript, however.)

\begin{definition}\label{def:equilibria}
	A point $x_e\in M$ is said to be an \textit{equilibrium point} of $(M,X)$ if $X(x_e) = 0\in T_{x_e}M$. 
\end{definition}
\begin{remark}
	The name comes from the fact that solutions starting at equilibria go nowhere: when $X(x_e)= 0$, then $\ph_{x_e}(t) \equiv x_e$ for all $t\in \R$. 
\end{remark}

\section{Lyapunov Stability}\label{sec:lyapunov}
We turn now to expositing Lyapunov stability theory, with moderately   abstract framing.  We suppose our manifolds are metric spaces, an assumption justified by the following elementary fact from geometry. 
\begin{fact}\label{fact:metrizability}
	Any second countable smooth manifold is metrizable (see, e.g., \cite[Corollary 13.30]{leeManifolds}). Metrizability means that the (a) metrizing metric induces the same topology as the original topology of the manifold. We will by default let $d_M:M\times M\rightarrow\R^{\geq 0}$ denote  such a metric on manifold $M$. 
\end{fact}

Because we are working with \textit{maps} of dynamical systems, we set notation to carry this information: 

\begin{notation} We will let $\scM_X$ denote the set of integral curves of vector field $X$: \begin{equation}\label{eq:setOfIntegralCurves} \scM_{X}\defeq \sF{DySys}\left(\big(\R,\frac{d}{dt}\big),\big(M,X\big)\right).\end{equation} As $X\in \Xf(M)$, specifying the vector field alone is sufficient for disambiguation. Leaving choice of target system open, $\scM_{(\cdot)} = \sF{DySys}\left(\big(\R,\frac{d}{dt}\big),\cdot\right).$ 
\end{notation}

\begin{remark}
	Fixing the state-space $M$ and ranging over dynamics $X\in \Xf(M)$---as opposed to varying inputs and outputs---is a perspective running through  work in \cite{lermannetworks}, \cite{lermanopennetworks}, and  \cite{lermanSchmidt1}. This viewpoint is related to and influenced by Willems' \textit{behavioral approach} to dynamical systems \cite{willems}, which fixes dynamical system (i.e.\ the vector field) and considers its collection \eqref{eq:setOfIntegralCurves}  of integral curves, itself suggestive of a  category-theoretic morphism-centric philosophy. 
\end{remark}

\begin{lemma}\label{lemma:inducedMetric}
	A metric $d_M:M\times M\rightarrow\R^{\geq 0}$ induces a pseudo-metric $\d_X:\scM_X^2\rightarrow \R^{\geq0}\cup\{\infty\}$   on the space $\scM_X$ of integral curves of $(M,X)$. 
\end{lemma}

\begin{remark}
	One may note that our definition in the following proof extends without much injury to any pair of curves $\g,\eta\in \Cc(\R,M)$ in $M$. In that setting, however, it is possible that $\d_X(\g,\eta) = 0 $ while  $\g \neq \eta$. Even restricted to $\scM_X$,  it is still possible for $\d_X(\ph_x,\ph_y) = \infty$, which is why we still call $\d_X$ a pseudo-metric. Our use for this ``metric'' will be to define a topology on $\scM_X$ in \cref{remark:topologyOnIntegralCurves}. 
\end{remark}

\begin{proof}
	Let $\ph_x,\ph_z\in \scM_X$ be two integral curves of $(M,X)$, namely unique solutions of $X$ with initial conditions $\ph_x(0) = x$ and $\ph_z(0)=z$. We define $$\d_X(\ph_x,\ph_z)\defeq \dissup_{t\geq 0} d_M(\ph_x(t),\ph_z(t)).$$ 
	Then  $$\begin{array}{ll} \dissup_{t\geq 0 }  d_M(\ph_x(t),\ph_z(t)) 
	& \leq \dissup_{t\geq 0} \big(d_M(\ph_x(t),\ph_y(t)) + d_M(\ph_y(t),\ph_z(t))\big) \\ &  \leq \dissup_{t\geq 0}d_M(\ph_x(t),\ph_y(t))+ \dissup_{t\geq 0}d_M(\ph_y(t),\ph_z(t)) \\ & = \d_X(\ph_x,\ph_y) + \d_X(\ph_y,\ph_z).\end{array} $$
	It is immediate that $\d_X(\ph_x,\ph_y)=\d_X(\ph_y,\ph_x)\geq 0$, with---by existence and uniqueness---equality when and only when $x=y$.\end{proof}

\begin{remark}\label{remark:topologyOnIntegralCurves} The pseudo-metric $\d_X$ induces a topology on $\scM_X$ generated by base of open sets \begin{equation}\label{eq:openBalls} B_\e(\ph_x)\defeq \big\{\ph_y\in \scM_X:\, \d_X(\ph_x,\ph_y)<\e\big\},\end{equation} for $\e>0$ and $x\in M$. When $B_\e(\ph_x)=\{\ph_x\}$ for each $\ph_x\in \scM_X$, the topology on $\scM_X$ is discrete. 
\end{remark}

We use this topology on $\scM_X$ to define Lyapunov stability:

\begin{definition}\label{def:lyapunovStabilityContinuity}
	Let $x_e\in M$ be an equilibrium point (\cref{def:equilibria}).  The point $x_e$ is said to be \textit{Lyapunov stable} if the solution map $\ph_{X,(\cdot)}:M\rightarrow\scM_X$ is continuous at $x_e$, with respect to the topology on $\scM_X$ defined in \cref{remark:topologyOnIntegralCurves}. (Recall from \cref{fact:metrizability} that the topology on $M$ agrees with that induced by metric $d_M$.)  
\end{definition}

\begin{remark} 
	This definition, while not ubiquitous in the literature (\cite{hespanha} provides one example of its existence), captures the notion that a solution which starts close to a stable equilibrium will remain nearby for all (positive) time.  A more  standard but equivalent definition of Lyapunov stability uses the $\d$-$\e$ criterion:   for any $\e>0$ there is a $\d_\e>0$ such that $\d_X\big(\ph_{X,x_e},\ph_{X,x_o}\big)<\e$ whenever $d_M(x_e,x_o)< \d_\e$, or even more explicitly: $$d_M(\ph_{X,x_e}(t),\ph_{X,x_o}(t))<\e\; \mbox{for all }t\geq0  \mbox{ whenever}\; d_M(x_e,x_o)<\d_\e.$$
\end{remark}
In fact, there is nothing sacrosanct about equilibria points in this definition:
\begin{definition}\label{def:stabilityContinuity}  We say that a point $x_o\in M $ is  \textit{stable} if the solution map $\ph_{X,(\cdot)}:M\rightarrow\scM_X$ is continuous at $x_o$, and that \textit{system} $(M,X)$ is stable when $\ph_{X}$ is stable for all $x_o\in M$. 
\end{definition}
\begin{remark}\label{remark:upshotsOfGeneralStability}
	There are two advantages of this definition.  First, it may apply to any arbitrary point (and therefore integral curve) of a dynamical system $(M,X)$.  For example, every point of dynamical system $(\R,\dx = -x)$ is stable in the sense of \cref{def:stabilityContinuity}.  Secondly, stability is not restricted to bounded solutions.  For example, every point of $(\R,\dx = 1)$ is stable, even though the solution $\ph_{X,x_0}(t) = x_0+t$ is unbounded.  Yet, in both cases, stability still captures a notion we care about: solutions which start close to each other remain close.   
\end{remark}

While stability makes sense and captures reasonable intuition for unbounded trajectories, we isolate our attention to bounded ones. We will rely on the following notion of boundedness for  \cref{prop:well-definedTopology} and \cref{lemma:pushForwardContinuousAtBounded}. 
 \begin{definition}\label{def:boundedPoints}
 	Let $(M,X)$ be a dynamical system and $x_0\in M$ a point.  We say that $x_0$ is \textit{bounded} if the trajectory defined by the solution $\ph_{X,x_0}$ is \textit{totally bounded} on $\R^{\geq0}$, i.e.\ if the set $\big\{\ph_{X,x_0}(t):\, t\geq 0\big\}$ is totally bounded  in $(M,d_M)$ (\cref{fact:metrizability}). We say that $\ph_{x_0}\in \scM_X$ is \textit{bounded} if $x_0$ is bounded. 
 \end{definition}
 \begin{remark}\label{remark:total-bounded} Total boundedness allows us to conclude that the closure $\el_{\ph_{x_0}}\defeq \overline{\{\ph_{X,x_0}(t):\, t\geq 0\}}$ is compact.  When $M = \R^n$, this condition distills simply as saying that the set $\ph_{x_0}\left(\R^{\geq0}\right)$ is bounded.  Observe that it may be possible for a bounded $\ph$ (as in \cref{def:boundedPoints}) to have unbounded image $\ph(\R)$ when considering the full domain.  This oddity justifies for us also using (just) `bounded' for $\ph$ to mean that $\ph(\R^{\geq0})$ is totally bounded.     \end{remark}

	 We would like to ensure that the topology $\scM_X$ does not depend on choice of metric $d_M$ from \cref{fact:metrizability}.  In other words, given metrics $d_M$, $d_M'$ both recovering the original topology of $M$, the pseudo-metrics $\d_X$ and $\d_X'$ (\cref{remark:topologyOnIntegralCurves}) induce the same topology on $\scM_X$. (Here, $\d'_X(\ph,\psi)\defeq \dissup_{t\geq 0 }d_M'(\ph(t),\psi(t))$.) We will prove and use  the following slightly weaker claim in \cref{sec:openMaps}: 
\begin{prop}\label{prop:well-definedTopology}
Let $d_M$ and $d_M'$ be two metrics on $M$ defining the same topology, and suppose that $\ph_x\in \scM_X$ is a bounded point (\cref{def:boundedPoints}). Then the induced pseudo-metrics $\d_X$ and $\d_X'$ (\cref{lemma:inducedMetric})  define equivalent local bases in $\scM_X$ at $\ph_x$.  
\end{prop}

\section{Open Maps Preserve Stability}\label{sec:openMaps}
 Recall that the maps of dynamical systems preserve integral curves (\cref{def:integralCurveAsMap}).  They also preserver equilibria.  Indeed, let $f:(M,X)\rightarrow (N,X)$ be a map of dynamical systems and $x_e\in M$ an equilibrium point.  Then  linearity of the differential implies that $Y(f(x_e))=Tf X(x_e)=Tf (0) =0$.   Alternatively, since maps of dynamical systems send  integral curves   to integral curves, $f_*\ph_{X,x_e}$ is a constant curve, so $0=\frac{d}{dt} f_*\ph_{X,x_e}(t)  = Y (\ph_{Y,f(x_e)}(t))$.

 Under some conditions on the map of systems, stability is also preserved.  Our main result, which we now state, provides such conditions. 
  \begin{theorem}\label{prop:pushStableToStable}\label{theorem:pushStableToStable}
  Let $(M,X)\xrightarrow{f}(N,Y)$ be a map of dynamical systems for which $f$ is open: $f(\Oc)$ is an open set in $N$ whenever $\Oc$ is open in $M$.  Suppose, further, that  $x_0\in M$ is stable and bounded (\cref{def:stabilityContinuity}, \cref{def:boundedPoints}).  Then $f(x_0)$ is stable. 	
  \end{theorem}
 
The proof of this theorem requires two lemmas, interesting in their own right. 
  \begin{lemma}\label{lemma:pushForwardContinuousAtBounded}
  	Let $f:(M,X)\rightarrow (N,Y)$ be a map of systems.  Then the pushforward $$\begin{array}{rrl} f_*: & \scM_X & \rightarrow\scM_Y\\ & \ph_{X,x} & \mapsto f\circ \ph_{X,x} = \ph_{Y,f(x)}\end{array}$$  is continuous at bounded points $\ph_{x_0}$ of $\scM_X$ (\cref{def:boundedPoints}). 
  \end{lemma}
  
The second one is needed  to prove \cref{lemma:pushForwardContinuousAtBounded}: 
\begin{lemma}\label{lemma:continuousDelta}
	Let $f:M\rightarrow N$ be a continuous map between manifolds and fix $\e>0$.  Then there is \textit{continuous} function $\d_\e:M\rightarrow\R^{>0}$ such that $d_M(x,x')<\d_\e(x)$ implies that $d_N(f(x),f(x'))< \e$.  
\end{lemma}

We call attention to our dual use of `$\d$' as both a metric on $\scM_X$ and as a function $M\xrightarrow{\d} \R^{\geq0}$.  In this proof, $\d$ and all its variants only refer to the latter function. 

\begin{proof}
	Since $f$ is continuous, there is \textit{a} function \begin{equation}\label{eq:tildeDelta} \tilde{\d}:M\rightarrow\R^{>0}\end{equation} (not necessarily continuous) such  that $d_M(x,x')<\tilde{\d}(x)$ implies that $d_N(f(x),f(x'))<\e/2$.  Let  $B_\d(x_0)\defeq \big\{x\in M:\, d_M(x_0,x)<\d\big\}$ and set \begin{equation}\label{eq:openCover}\scB \defeq \left\{B_{\tilde{\d}(x_0)/2}(x_0):\, x_0\in M_0\right\}\end{equation} to be a locally finite open cover of $M$,  with  $M_0\subset M$. Then there  is smooth partition of unity $\big\{\r_{x_0}:M\rightarrow[0,1]	:\, x_0\in M_0\big\}$  subordinate to $\scB$ (see, e.g., \cite[Theorem 13.7]{tu}). 
	
	We  define function $\d_\e:M\rightarrow\R^{>0}$ by \begin{equation}\label{eq:DefDelta} \d_\e(\cdot) \defeq \frac{1}{2}\diss_{x_0\in M_0}\tilde{\d}(x_0)\r_{x_0}(\cdot),\end{equation}  which is smooth, and therefore continuous,  since each $\r_{x_0}(\cdot)$ is. We must  verify that this $\d_\e$  satisfies the delta-epsilon constraint, namely that $d_N(f(x),f(x'))<\e$ whenever $d_M(x,x')<\d_\e(x)$.
	
	Let $x\in M$ and set \begin{equation}\label{eq:maxBall}\begin{array}{ll} \ox& \defeq \displaystyle\arg\max_{x_0\in M_0}\left\{\tilde{\d}(x_0):\, \r_{x_0}(x)\neq 0\right\},\\\overline{\d}& \defeq \tilde{\d}(\ox)= \dismax_{x_0\in M_0}\big\{\tilde{\d}(x_0):\, \r_{x_0}(x)\neq 0\big\}.\end{array}\end{equation} By definition, $x\in B_{\overline{\d }/2}(\ox)$  (for pictorial reference, consider \cref{fig:fig}) and   $\d_\e(x)\leq \frac{1}{2}\overline{\d}$ (\cref{eq:DefDelta}). 
	
	Now, suppose   that $d_M(x,x')<\d_\e(x)$. We observe, first, that $$d_N(f(x),f(x'))\leq d_N(f(x),f(\ox)) + d_N(f(\ox),f(x'))$$ because $d_N$ is a metric. We bound each term by $\e/2$.  On $M$, \begin{equation}\label{eq:epsilonWhatevs} d_M(x,x')\leq  2d_M(x,\ox) + d_M(x,x')= d_M(x,\ox) + \big(d_M(\ox,x)+d_M(x,x')\big).\end{equation}   Since $\rho_{\ox} (x)\neq 0$ and $\mbox{supp}(\rho_{\ox} )\subseteq B_{\overline{\d}/2}(\ox)$, we conclude that $d_M(x,\ox)<\frac{1}{2}\overline{\d} = \frac{1}{2}\tilde{\d}(\ox)<\tilde{\d}(\ox)$ which implies (\cref{eq:openCover}, \eqref{eq:maxBall}) that \begin{equation}\label{eq:epsilon1} d_N(f(x),f(\ox))<\e/2.\end{equation}  Similarly,  $\big( d_M(\ox,x)+d_M(x,x')\big)< \frac{1}{2}\overline{\d}+  \frac{1}{2}\overline{\d} = \overline{\d} = \tilde{\d}(\ox)$  which implies, since $d_M(\ox,x')\leq d_M(\ox,x)+d_M(x,x')$, that \begin{equation}\label{eq:epsilon2} d_N(f(\ox),f(x'))< \e/2.\end{equation}  Together, inequalities eq.\ \eqref{eq:epsilon1} and eq.\ \eqref{eq:epsilon2}  imply that $$d_N(f(x),f(x')) < \e,$$ as desired. \end{proof}
\begin{figure}[h!]
\centerline{\includegraphics[scale = 0.4]{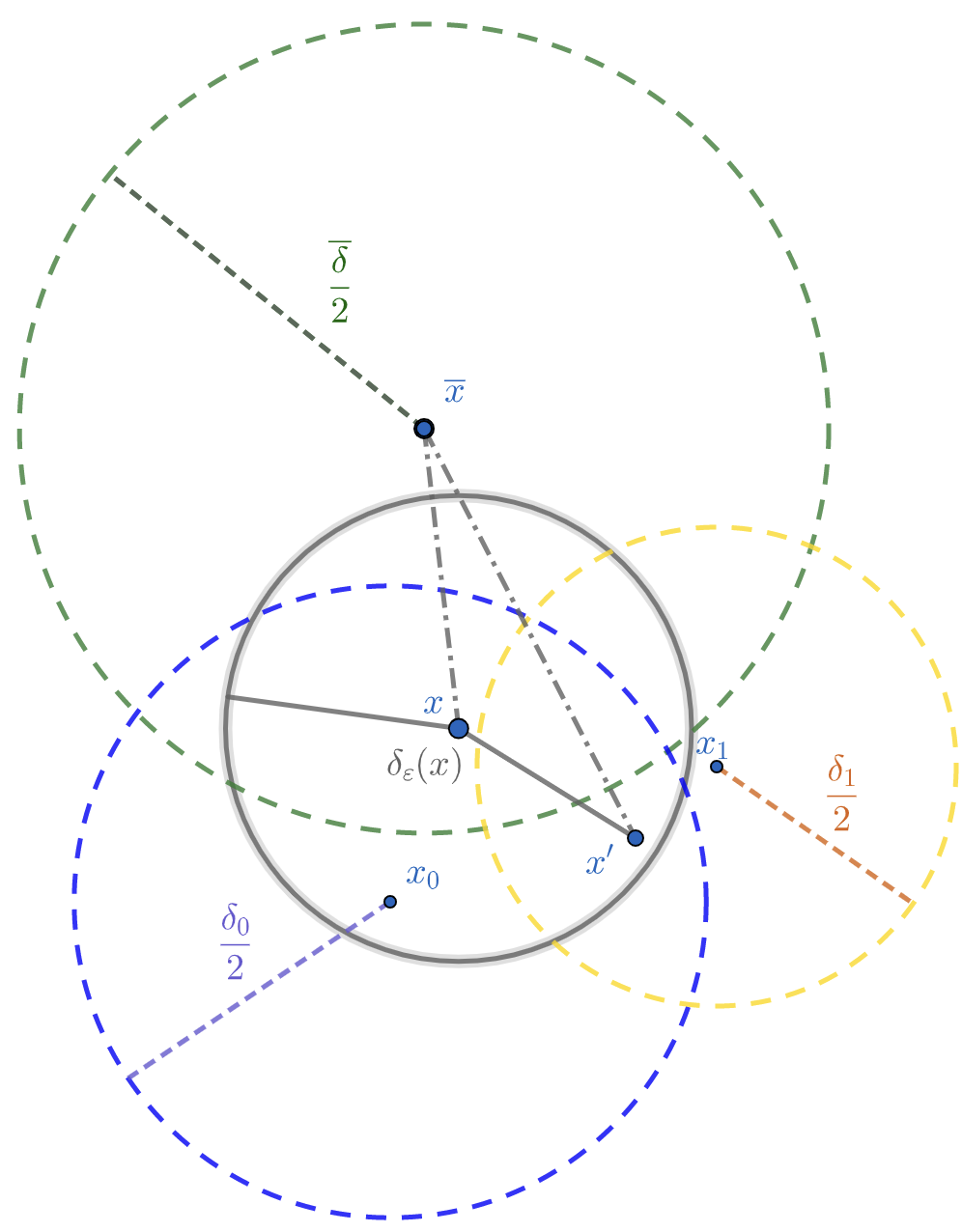}}\caption{Construction of $\d_\e(\cdot)$}\label{fig:fig}	
\end{figure}

Before continuing with the proof of \cref{lemma:pushForwardContinuousAtBounded}, we establish that  local base of topology of $\scM_X$ is well-defined at bounded points (\cref{def:boundedPoints}). (In this proof, we revert to using $\d$ as a pseudo-metric on $\scM_X$.) 
 \begin{proof}[Proof of \cref{prop:well-definedTopology}]
	Fix  bounded $\ph_x\in \scM_X$,  $\e>0$, and open ball $B_\e(\ph_x)\subset \scM_X$ \eqref{eq:openBalls}. We show that there is open ball $B_{\e_*'}'(\ph_x)=\big\{\ph\in \scM_X:\, \d_X'(\ph_x,\ph)<\e_*'\big\}\subset B_\e(\ph_x)$ for some $\e_*'>0$. Because $d_M$ and $d_M'$ define the same topology on $M$, the identity map $(M,d_M')\xrightarrow{id_M}(M,d_M)$ is continuous, and \cref{lemma:continuousDelta} implies, therefore, that there is continuous $\e':M\rightarrow\R^{>0}$ for which $d_M(x,x')<\e$ whenever $d_M'(x,x')<\e'(x)$. Since $\ph_x$ is bounded, the closure $\el_\ph\defeq\overline{\ph_x(\R^{\geq0})}$ is compact (\cref{remark:total-bounded}), so $x_*\defeq\disl_{t\rightarrow\infty}\ph_x(t) \in \el_\ph$, and $\disl_{t\rightarrow\infty}\e'(\ph(t)) = \e'(x_*) > 0$. Therefore, continuity of $\e'$, compactness of $\el_\ph$, and  the inclusion $\e'(\el_\ph)\subset\R^{>0}$ imply that  $\e'_*\defeq\disinf_{t\geq 0}\e'(t)>0$. It readily follows that $B_{\e'_*}'(\ph_x) \subset B_\e(\ph_x)$.  \end{proof}
  
  The proof of \cref{lemma:pushForwardContinuousAtBounded} is similar.   \begin{proof}[Proof of \cref{lemma:pushForwardContinuousAtBounded}]  	 Fix $\e>0$ and let $\ph\in \scM_X$ be bounded (\cref{def:boundedPoints}). We  show that there is $\hat{\d}>0$ for which $f_*\left(B_{\hat{\d}}(\ph)\right)\subseteq B_{\e }(f_*\ph)$.  Let $\d_{\e}:M\rightarrow \R^{>0}$ be a continuous function  satisfying delta-epsilon condition for $\e $ (\cref{lemma:continuousDelta}), so that $$d_N\left(f\big(\ph(t)\big),f(x)\right)<\e $$ whenever   $$d_M(\ph(t),x)<\d_{\e }(\ph(t)).$$  Because  $\Lc_\ph\defeq  \overline{\left\{\ph(t):\, t\geq 0 \right\}}$ is compact (\cref{remark:total-bounded}),  $$x_*\defeq \disl_{t\rightarrow\infty}\ph(t) \in \el_\ph\subset M$$ and because  $\d_{\e }(\cdot)$ is continuous, $$\disl_{t\rightarrow\infty}\d_{\e }(\ph(t)) = \d_{\e }\left(x_*\right) > 0.$$ Compactness of $\el_\ph$ and continuity of $\d_{\e}$ imply that the minimum $\hat{\d}\defeq \dismin\big\{\overline{\d}\in \d_{\e }(\el_\ph)\big\}$ is obtained, and because $\d_{\e }(\el_\ph)\subset\R^{>0}$, $\hat{\d}>0$. We thus conclude that $$f_*(B_{\hat{\d}}(\ph))\subseteq B_{\e }(f_*\ph), $$ 	as required
  \end{proof}

    Finally, the proof of \cref{theorem:pushStableToStable}:
  \begin{proof}[Proof of \cref{prop:pushStableToStable}]
	Suppose that $x_0\in M$ is stable and bounded.  To show that $f(x_0)\in M$ is stable, we must show that the solution map $\ph_Y:N\rightarrow \scM_Y$ is continuous at $f(x_0)$. Let $\Oc\subseteq \scM_Y$ be open containing $\ph_{Y,f(x_0)}$ and consider the commutative diagram $$\begin{tikzcd}[column sep = large, row sep= large]
	M\arrow[r,"\ph_X"]\arrow[d,"f"] & \scM_X \arrow[d,"f_*"]\\ N\ar[r,"\ph_Y"] & \scM_Y.
	\end{tikzcd}$$  
Since $\ph_Y\circ f = f_*\circ \ph_X$, we have that $f^{-1}\circ \ph_Y^{-1} (\Oc) = \ph_X^{-1}\circ f_*^{-1}(\Oc)$ and therefore $$\ph_Y^{-1}(\Oc)\supseteq f \left(f^{-1}\big(\ph_Y^{-1}(\Oc)\big)\right)  = f\left(\ph_X^{-1}\big(f_*^{-1}(\Oc)\big)\right),$$ which is open because $f_*$ is continuous (\cref{lemma:pushForwardContinuousAtBounded}), $\ph_X$ is continuous at $x_0$ by assumption (\cref{def:stabilityContinuity}), and $f$ is open. \end{proof}

  \begin{example}
  	Consider the nonlinear dynamical system $(\R,\dx = -x^3)$ and map of systems $$\begin{tikzcd}
\R^{>0}\arrow[d,"-x"]\arrow[r,"f"] & \R \arrow[d,"-x^3"]\\
T\R^{>0}\arrow[r, "Tf"] & T\R, 	
\end{tikzcd}$$ where  $$f(x) = \frac{1}{\sqrt{\log\left(\frac{1}{x^2}\right)+1}}.$$ 

As  observed previously, $(\R,\dx = -x)$ is stable and $f$ is open at $x=1$. Therefore  \cref{prop:pushStableToStable} implies that $f(1)=1$ is stable in $(\R,\dx = -x^3)$.  \end{example}

  \begin{example}
  	Consider  system $\dx = 1$, whose solution is given by $x(t) = x_0 + t$.  While obviously stable  (\cref{remark:upshotsOfGeneralStability}), we observe this fact as a result of \cref{theorem:pushStableToStable}.  Consider map of systems $$\begin{tikzcd}
	 \R^{>0}\arrow[r,"-\log(x)"]\arrow[d,"-x"] & \R\arrow[d,"1"]\\
T\R^{>0}\arrow[r,"-\frac{1}{x}"] & T\R,  
\end{tikzcd}$$ which is an open map.  Since $\dx = -x$ is stable, we conclude that $\dx = 1$ is as well.  

  \end{example}

\addcontentsline{toc}{chapter}{Bibliography}
\bibliographystyle{plain}
\bibliography{phdref}

\end{document}